\documentclass[twocolumns]{autart}
\usepackage{amsmath}
\usepackage{mathtools}
\usepackage{amssymb}
\usepackage{here}
\usepackage{mathrsfs}
\usepackage{multirow}
\usepackage{cite}
\usepackage{dsfont}
\usepackage{color}

\renewcommand{\theenumi}{\alph{enumi}}

\everymath{\displaystyle}

\newtheorem{theorem}{Theorem}

\newtheorem{proposition}[theorem]{Proposition}
\newtheorem{lemma}[theorem]{Lemma}
\newtheorem{define}[theorem]{Definition}

\newtheorem{remark}[theorem]{Remark}

\newtheorem{problem}[theorem]{Problem}

\newcommand{\mendth}{\hfill \ensuremath{\vartriangle}}

\newcommand{\mendprop}{\hfill \ensuremath{\triangledown}}

\DeclareMathOperator*{\esssup}{ess\ sup}

\makeatletter
\newcommand{\pushright}[1]{\ifmeasuring@#1\else\hfill \displaystyle#1\fi\ignorespaces}
\newcommand{\pushleft}[1]{\ifmeasuring@#1\else\omit$\displaystyle#1$\hfill\fi\ignorespaces}
\makeatother

\newenvironment{proof}{{\it Proof :~}}{\hfill$\diamondsuit$\\}

\begin{document}


\begin{frontmatter}

\title{Interval peak-to-peak observers for continuous- and discrete-time systems with persistent inputs and delays}

\author{Corentin Briat} \ead{corentin.briat@bsse.ethz.ch, corentin@briat.info}\ead[url]{http://www.briat.info} and
\author{Mustafa Khammash}\ead{mustafa.khammash@bsse.ethz.ch}\ead[url]{https://www.bsse.ethz.ch/ctsb}

\address{Department of Biosystems Science and Engineering, ETH--Z\"{u}rich, Switzerland.}


\begin{abstract}
While the design of optimal peak-to-peak controllers/observers for linear systems is known to be a difficult problem, this problem becomes interestingly much easier in the context of interval observers because of the positive nature of the error dynamics. Indeed, by exploiting several recent results on positive systems, we propose a novel and non-conservative approach  formulated in terms of tractable finite-dimensional linear programs for designing a class of interval observers achieving minimum peak-to-peak gain. The optimal observer is notably shown to be uniform over the set of all possible mappings between observation errors and their weighted versions, which parallels a recent result on the stabilization of linear positive systems. Results pertaining on the interval observation of time-delay and discrete-time systems are then obtained as a direct application of the proposed method, emphasizing then its versatility. Several examples on the interval observation of linear and nonlinear systems are finally given for illustration.
\end{abstract}
\end{frontmatter}

\section{Introduction}

Interval observers \cite{Gouze:00, Rapaport:03} are an interesting class of observers that can be used to estimate upper- and lower-bounds on the state value of a given system at any given time. The rationale for using such observers lies in the fact that when the observed system is subject to disturbances or uncertainties, estimation errors are often unavoidable, unless decoupling is possible. When this happens, then it becomes unclear in which region the state currently lies. This problem is elegantly resolved, among others, by the use of interval observers that will estimate an upper- and a lower-bound on the current state-value, thereby explicitly defining an interval where the state lies within. In this respect, interval observers are highly relevant in the context of observation of systems for which only a poor model is available, like in fields such as ecology, epidemiology or biology; see e.g. \cite{Murray:02}. Interval observers have been already considered for systems with inputs \cite{Mazenc:11}, linear systems \cite{AitRami:11b,Mazenc:12,Cacace:15}, uncertain systems \cite{Bolajraf:15}, time-varying systems \cite{Thabet:14}, delay systems \cite{Li:12,Polyakov:13,Efimov:13c}, nonlinear systems \cite{Raissi:10,Efimov:13}, LPV systems \cite{Efimov:13b,Chebotarev:15}, discrete-time systems \cite{Mazenc:13}, for stabilization \cite{Polyakov:13,Efimov:15b}, etc.

The underlying idea is to design an observer in a way that makes the observation errors nonnegative over time, which is equivalent to having error dynamics represented by a positive system \cite{Farina:00}. This positivity property is very important to consider because of the interesting properties of positive systems. It is indeed now well-known that their stability \cite{Haddad:05} and their $L_1$-/$L_\infty$-gain can be established using finite-dimensional linear programs \cite{Briat:11g,Briat:11h}. 

We hence consider here a simple class of interval observers and propose a methodology for designing them so that the $L_\infty$-gain of the operator mapping the persistent input errors (disturbances) to the observed output -- a weighted version of the observation errors -- is minimum; see e.g. \cite{Voulgaris:95,Blanchini:94}. By doing so, we simply attempt to minimize the influence of the disturbances on the width of the interval in a peak-to-peak sense. It is also important to stress that, in this context, considering bounded and persistent disturbances (i.e. in $L_\infty$) seems more relevant than considering vanishing ones, as in the $L_1$ and in the $L_2$ frameworks, as the disturbances may depend on persistent exogenous signals or on the state of the process. The general control/observation/filtering problems are reputed to be difficult in this framework; see e.g. \cite{Vidyasagar:86,Dahleh:87,Dahleh:95,Voulgaris:95,Blanchini:94}. The approach proposed in this paper is different than the one that could have been based on the so-called $*$-norm \cite{Nagpal:94} which is known to be a loose characterization of the peak-to-peak gain as minimizing the $*$-norm does not necessarily result in a small peak-to-peak norm; see e.g.\cite{Venkatesh:95}. The approach is also different from the one considered in the context of set-valued observers \cite{Shamma:99} that aims at estimating  the set of all possible states based on the knowledge of the output of the system and a model of the exogenous signals. These observers are known to be difficult to implement in real-time due to a high computational burden.

The approach considered in this paper is instead based on positive systems theory and leads to both constructive and nonconservative design conditions that take the form of tractable finite-dimensional linear programs, thereby circumventing the inherent difficulties of the design of general optimal peak-to-peak observers. Interestingly, the optimal interval observer is uniform over all the possible values for the matrix parameter that maps the observation errors to the performance outputs. This essentially means that changing the performance specifications (i.e. the performance output) will not require the re-design of the optimal interval-observer. Note, however, that the performance level, i.e. the actual $L_\infty$-gain, is not uniform and will actually vary when the performance output is changed. This property is dual to the one obtained in \cite{Ebihara:12} that states that the optimal state-feedback controller that minimizes the $L_1$-gain of the closed-loop system is uniform over the set of all the possible values for the input matrices driving the disturbance into the system. Finally, it seems important to mention that even though the class of problems we consider in this paper may seem restrictive, it is demonstrated that the proposed approach extends, among others, to more complex observers, to systems with delays and to discrete-time systems in a very natural way.





\textbf{Outline.} Section \ref{sec:prel} introduces the considered norms for signals and systems, as well as tools for computing them whenever the considered systems are positive. The problem of designing $L_\infty$-to-$L_\infty$ interval-observers for continuous-time systems with persistent disturbances is solved in Section \ref{sec:obs}. Examples are discussed in Section \ref{sec:ex}.

\textbf{Notations.} The set of positive and nonnegative real numbers are denoted by $\mathbb{R}_{>0}$ and $\mathbb{R}_{\ge0}$, respectively, and they naturally extend to vectors and matrices as $\mathbb{R}_{\ge0}^{n\times m}$, $\mathbb{R}_{>0}^{n\times m}$ where the inequality signs are entry-wise. The $n$-dimensional vector of ones is denoted by $\mathds{1}_n$. For two vectors $a,b\in\mathbb{R}^n$, $a<(\le)b$ means that $b-a\in\mathbb{R}^n_{>0}$ $(\mathbb{R}^n_{\ge0})$. The sequence $\{e_i\}_{i=1}^n$ is used to represent the natural basis of $\mathbb{R}^n$.

\section{Preliminaries}\label{sec:prel}

\subsection{The $L_\infty$-norm and the $L_\infty$-gain of linear systems}

\begin{define}[\cite{Desoer:75a}]
  The space of essentially bounded functions $w:\mathbb{R}_{\ge0}\to\mathbb{R}^p$ is denoted by $L_\infty(\mathbb{R}_{\ge0},\mathbb{R}^p)$ and the associated norm is defined as
  \begin{equation*}
    ||w||_{L_\infty}:=\esssup_{t\ge0}||w(t)||_\infty
  \end{equation*}
  where $||\cdot||_\infty$ is the standard vector $\infty$-norm. In what follows, we will use the shorthand $L_\infty$ for simplicity.\mendprop
\end{define}

\begin{define}[\cite{Desoer:75a}]
  Let us consider an asymptotically stable LTI system $G$ mapping $p$ inputs to $q$ outputs. Then its $L_\infty$-gain (or $L_1$-norm) is defined as
  \begin{align*}
   & ||G||_{L_\infty-L_\infty}=\sup_{||w||_{L_\infty}=1}||Gw||_{L_\infty}.&\pushright{\triangledown}
  \end{align*}
\end{define}

\subsection{Linear positive systems}

Let us start with fundamental definitions:
\begin{define}
  A matrix $A\in\mathbb{R}^{n\times n}$ is said to be Metzler if all its off-diagonal entries are nonnegative.\mendprop
\end{define}
\begin{proposition}[\cite{Farina:00}]
  Let us consider the following system
    \begin{equation}\label{eq:syst}
      \begin{array}{lcl}
        \dot{x}(t)&=&Ax(t)+Ew(t),\ x(0)=x_0\\
         z(t)&=&C_zx(t)+F_zw(t).
       \end{array}
    \end{equation}
    where $x,x_0\in\mathbb{R}^n$, $w\in\mathbb{R}^p$,  and $z\in\mathbb{R}^q$ are the state of the system, the initial condition, the input and the output. The following statements are equivalent:
    \begin{enumerate}
      \item The system is positive, i.e. for all $x_0\ge0$ and $w(t)\ge0$, we have that $x(t)\ge0$ and $z(t)\ge0$.
      \item The matrix $A$ is Metzler and the matrices $E,C_z$ and $F_z$ are nonnegative matrices.\mendth
    \end{enumerate}
\end{proposition}

We now state two fundamental stability results regarding the stability of linear positive systems:
\begin{theorem}[\cite{Farina:00}]\label{th:stab}
  Let us consider a Metzler matrix $A\in\mathbb{R}^{n\times n}$. Then, the following statements are equivalent:
  \begin{enumerate}
    \item $A$ is Hurwitz stable.
    \item There exists a $\lambda\in\mathbb{R}^n_{>0}$ such that $\lambda^TA<0$.
    \item There exists a $\mu\in\mathbb{R}^n_{>0}$ such that $A\mu<0$.\mendth
  \end{enumerate}
\end{theorem}
\begin{theorem}[\cite{Briat:11g,Briat:11h}]\label{th:Linf}
  Let $A\in\mathbb{R}^{n\times n}$ be a Metzler matrix and $E\in\mathbb{R}^{n\times p}_{\ge0}$, $C\in\mathbb{R}^{q\times n}_{\ge0}$ and $F\in\mathbb{R}^{q\times p}_{\ge0}$ be nonnegative matrices. Let $\gamma>0$, then the following statements are equivalent:
  \begin{enumerate}
    \item The linear system \eqref{eq:syst} is asymptotically stable and $||w\mapsto z||_{L_\infty-L_\infty}<\gamma$.
    %
    %
    \item $A$ is Hurwitz stable and $(-C_zA^{-1}E+F_z)\mathds{1}_p<\gamma \mathds{1}_q$.
    \item There exists a $\lambda\in\mathbb{R}_{>0}^n$ such that the conditions
    \begin{equation}\label{eq:Linfcond}
      \begin{array}{rcl}
        A\lambda+E\mathds{1}_p&<&0\\
        C_z\lambda+F_z\mathds{1}_p-\gamma\mathds{1}_q&<&0
      \end{array}
    \end{equation}
    hold.\mendth
  \end{enumerate}
\end{theorem}


\section{A class of $L_\infty$-to-$L_\infty$ interval-observers}\label{sec:obs}

\subsection{Preliminaries}

Let us consider now the following system
\begin{equation}\label{eq:systy}
  \begin{array}{rcl}
    \dot{x}(t)&=&Ax(t)+Ew(t),\ x(0)=x_0\\
    y(t)&=&C x(t)+F w(t)
  \end{array}
\end{equation}
where $x,x_0\in\mathbb{R}^n$, $w\in\mathbb{R}^p$, $y\in\mathbb{R}^r$ are the state of the system, the initial condition, the persistent disturbance input and the measured output. Note that this system is not necessarily positive. We are interested in finding an interval-observer of the form
\begin{equation}\label{eq:obs}
\begin{array}{rcl}
      \dot{x}^\bullet(t)&=&Ax^\bullet(t)+Ew^\bullet(t)+Lm^\bullet(t)\\
      m^\bullet(t)&=&y(t)-C x^\bullet(t)-F w^\bullet(t)\\
      x^\bullet(0)&=&x_0^\bullet
\end{array}
\end{equation}
where $\bullet\in\{-,+\}$.
Above, the observer with the superscript ``$+$'' is meant to estimate an upper-bound on the state value whereas the observer with the superscript ``-'' is meant to estimate a lower-bound; i.e. we would like to have $x^-(t)\le x(t)\le x^+(t)$ for all $t\ge0$ provided that $x_0^-\le x_0\le x_0^+$. The signals $w^-,w^+\in L_\infty(\mathbb{R}_{\ge0},\mathbb{R}^p)$ are the lower- and the upper-bound on the disturbance $w(t)$, i.e. $w^-(t)\le w(t)\le w^+(t)$ for all $t\ge0$. We then accordingly define the errors  $e^+(t):=x^+(t)-x(t)$ and $e^-(t):=x(t)-x^-(t)$ that are described by the models
\begin{equation}\label{eq:error}
\begin{array}{rcl}
    \dot{e}^\bullet(t)&=&(A-LC )e^\bullet(t)+(E-LF )\delta^\bullet(t)\\
    \zeta^\bullet(t)&=&M^\bullet e^\bullet
\end{array}
\end{equation}
where $\bullet\in\{-,+\}$ and where $\delta^+(t):=w^+(t)-w(t)\in\mathbb{R}_{\ge0}^p$, $\delta^-(t):=w(t)-w^-(t)\in\mathbb{R}_{\ge0}^p$ are the input errors. The matrices $M^-,M^+\in\mathbb{R}^{q\times n}_{\ge0}$ are the nonzero matrices driving the observation errors $e^\bullet$ to the observed outputs $\zeta^\bullet$, that we assume to be chosen a priori.

\begin{remark}
  It is important to stress here that the choice of a common observer gain $L$ in \eqref{eq:obs} seems to be restrictive at first sight. This is quite surprisingly not the case as it will be shown later that the same optimal gain will simultaneously minimize the $L_\infty$-gain of both the transfers $\delta^-\mapsto\zeta^-$ and $\delta^+\mapsto\zeta^+$, even  when $M^-\ne M^+$. In this regard, the values of the matrices $M^-$ and $M^+$ do not matter at all when designing the optimal gain $L$. This will be discussed in more details in Section \ref{sec:optL1obs}.
\end{remark}

With all these elements in mind, we can state the considered observation problem:
\begin{problem}\label{problem}
  Find an interval observer of the form \eqref{eq:obs} such that
  \begin{enumerate}
    \item The linear systems in \eqref{eq:error} are positive, i.e. $A-LC $ is Metzler and $E-LF $ is nonnegative;
    \item The linear systems in \eqref{eq:error} are asymptotically stable; i.e. $A-LC $ is Hurwitz stable;
    \item The $L_\infty$-gains of the transfers $\delta^\bullet\mapsto\zeta^\bullet$, $\bullet\in\{-,+\}$, are minimal.
  \end{enumerate}
\end{problem}

\begin{remark}
More complex observer structures can also be considered in order to reduce the conservatism of the approach. For instance, high-order observers such as the ones proposed in \cite{Blanchini:12} may be considered. Alternatively, a relaxed version of the proposed one which takes the form
\begin{equation}\label{eq:relaxobs}
  \begin{array}{lcl}
  \dot{\tilde{x}}^{+}(t)&=&[A-LC]\tilde{x}^{+}(t)+Ly(t)\\
&&+[E-LF]^{+}w^{+}(t)-[E-LF]^{-}w^{-}(t),\\
\dot{\tilde{x}}^{-}(t)&=&[A-LC]\tilde{x}^{-}(t)+Ly(t)\\
&&+[E-LF]^{+}w^{-}(t)-[E-LF]^{-}w^{+}(t)
  \end{array}
\end{equation}
where $[E-LF]^{+}=\max\{0,E-LF\}$ and $[E-LF]^{-}=[E-LF]^{+}-[E-LF]$ can also be considered. The advantage of the above one lies in the fact that the nonnegativity of the matrix $E-LF$ is not required anymore. However, when an interval observer of the form \eqref{eq:obs} exists, it may be preferable to use it since it captures more information about the system. This will be illustrated in the examples treated in Section \ref{sec:ex}.
\end{remark}

\begin{remark}
  As it may be sometimes difficult to find a matrix $L$ that makes the matrix $A-LC$ both Metzler and Hurwitz stable, the determination of  a nonsingular matrix $P$ and a matrix $L$ such that the matrix $P(A-LC)P^{-1}$ is Metzler and Hurwitz stable is often considered. Suitable pairs $(P,L)$ can be computed using various methods; see e.g. \cite{Raissi:12}.
\end{remark}


%

\subsection{The uniformity of $L_\infty$-to-$L_\infty$ interval observers}\label{sec:optL1obs}

For the sake of generality, let us consider now the following slightly more general system:
    \begin{equation}\label{eq:syst2}
      \Sigma:\left\{\begin{array}{lcl}
      \dot{\xi}(t)&=&(A -LC  )\xi(t)+(E - LF)\omega(t)\\
        \chi(t)&=&M\xi(t)+N\omega(t)
      \end{array}\right.
    \end{equation}
    where $\xi\in\mathbb{R}^{n}$, $\omega\in\mathbb{R}_{\ge0}^{p}$, $\chi\in\mathbb{R}^{q}$ and $L\in\mathbb{R}^{n\times r}$. The matrices are defined such that their dimension corresponds to those of the signals. The only assumption is that both $ M$ and $N$ are nonnegative and that $M$ is nonzero. We now address the question stated in Problem \ref{problem} where the matrix $N$ is now considered and $(x,\delta^\bullet,\zeta^\bullet)$ are replaced by $(\xi,\omega,\chi)$.

     Formally speaking, we are interested in solving the problem
    \begin{equation}
      \gamma^*=\inf_{L \in\mathcal{L}}||\omega\mapsto\chi||_{L_\infty-L_\infty}
    \end{equation}
    where
    \begin{equation}
      \mathcal{L}:=\left\{ L \in\mathbb{R}^{n\times r}\left|\begin{array}{c}
        A-LC \ \textnormal{Metzler and Hurwitz},\\
         E-LF\ge0
      \end{array}\right. \right\}
    \end{equation}
    and in computing the matrix $L^*$ that achieves this infimum. To this aim, let us define $\gamma^*:=\gamma^*(\Sigma(M  ,N ))$ and $L^*:=L^*(\Sigma(M  ,N  ))$. The main rationale for the notation $\Sigma(M ,N)$ stems from the property of uniformity of the optimal $L_\infty$-to-$L_\infty$ interval observers stated in the result below:
    \begin{lemma}\label{lem:SISO}
    Assume that the system  \eqref{eq:syst2} is MISO; i.e. $q=1$. Then, the membership relation
      \begin{equation}\label{eq:Lstar}
    L^*(\Sigma(\mathds{1}_{n}^T,\mathds{1}_{p}^T))\in\mathcal{L}^*(\Sigma(M,N))
      \end{equation}
      holds for all $(M,N)\in\mathbb{R}^{1\times n}_{\ge0}\times\mathbb{R}^{1\times p}_{\ge0}$ where
          \begin{equation}
      \mathcal{L}^*(\Sigma(M,N )):=\left\{L\in\mathcal{L}:||\Sigma(M,N)||_{L_\infty-L_\infty}=\gamma^*\right\}.
    \end{equation}
    \end{lemma}
    \begin{proof}
          The proof uses the same arguments as for proving Lemma 4 in \cite{Ebihara:12} and is omitted for brevity.
    \end{proof}

    The above result says that, for a MISO system of the form \eqref{eq:syst2}, there exists a value for $L$ that simultaneously minimizes the $L_\infty$-gain of the system $\Sigma(M,N  )$ for all matrices $(M  ,N)\in\mathbb{R}^{1\times n}_{\ge0}\times\mathbb{R}^{1\times p}_{\ge0}$. This value can be, moreover, calculated by considering the system $\Sigma(\mathds{1}_{n}^T,\mathds{1}_{p}^T)$ and Theorem \ref{th:Linf}. The issue here is that this result only holds for MISO systems but can be generalized using the fact that the $L_\infty$-gain of $\Sigma$ can be expressed as $\textstyle\max_{i=1,\ldots,q}\{||\Sigma_i||_{L_\infty-L_\infty}\}$, where $\Sigma_i$ maps $\omega$ to $\chi_i$. However, Theorem \ref{th:Linf} cannot be simply considered for all the $\Sigma_i$'s independently as each subsystem $\Sigma_i$ would require its own Lyapunov function; i.e. the use of different $\lambda$'s. A simple solution is to enforce the $\lambda$'s to be all the same, which is equivalent to searching for a common Lyapunov function for all the subsystems and is known to be a conservative procedure in general. This is luckily not the case here as proved in the next two results:
    %
    \begin{lemma}\label{th:rankone}
        Let $W\in\mathbb{R}^{n\times n}$ be a Metzler and Hurwitz stable matrix and let $u,v_i\in\mathbb{R}^{n}_{\ge0}$, $i=1,\ldots,q$, be nonnegative vectors. Then, the following statements are equivalent
        \begin{enumerate}
          \item The matrices $W+uv_i^T$ are Hurwitz stable for all $i=1,\ldots,q$.
          \item There exists a $\psi\in\mathbb{R}_{>0}^{n}$ such that $(W+uv_i^T)\psi<0$ holds for all $i=1,\ldots,q$.
        \end{enumerate}
    \end{lemma}
    \begin{proof}
       This result is dual to the one obtained in \cite{Ebihara:12} and its proof, based on Farkas' Lemma, is omitted for brevity.
    \end{proof}
\begin{lemma}\label{lem:glab2}
Let $\gamma\in\mathbb{R}_{>0}$ and $L\in\mathcal{L}$. Then, the following statements are equivalent:
\begin{enumerate}
  \item $||\Sigma(M,N)||_{L_\infty-L_\infty}<\gamma$.
  \item The  matrices
  \begin{equation}
\begin{bmatrix}
     A-LC  & (E-LF)\mathds{1}_p\\
      \e_i^TM & e_i^TN-\gamma
    \end{bmatrix},\ i=1,\ldots,q
  \end{equation}
  are Metzler and Hurwitz stable.
  \item $||\Sigma(e_i^TM,e_i^TN)||_{L_\infty-L_\infty}<\gamma$,\ $i=1,\ldots,q$.
\end{enumerate}
\end{lemma}
\begin{proof}
  The proof follows the same lines as the proof of Lemma 6 in \cite{Ebihara:12}, it is therefore only sketched. The equivalence between (b) and (c)  follows from Theorem \ref{th:stab}. The implication (a)$\Rightarrow$(b) is straightforward whereas the reverse implication follows from an application of Lemma \ref{th:rankone}.
\end{proof}

%

\subsection{Design of $L_\infty$-to-$L_\infty$ interval observers}

By combining the results stated in Theorem \ref{th:Linf}, Lemma \ref{lem:SISO} and Lemma \ref{lem:glab2}, we can state the main result of the paper:
\begin{theorem}\label{th:L1obs}
The following statements are equivalent:
\begin{enumerate}
  \item There exists an optimal $L_\infty$-to-$L_\infty$ interval-observer of the form \eqref{eq:obs} for the system \eqref{eq:systy}.
  %
  \item We have that
  \begin{equation}\label{eq:uni1}
     L^*(\Sigma(\mathds{1}_{n}^T,\mathds{1}_{p}^T))\in\mathcal{L}^*(\Sigma(M,N))
      \end{equation}
  for all $(M,N)\in\mathbb{R}^{q\times n}_{\ge0}\times\mathbb{R}^{q\times p}_{\ge0}$. Moreover, the optimal gain is given by $L^*(\Sigma(\mathds{1}_{n}^T,\mathds{1}_{p}^T))=(X^*)^{-1}U^*$ where the pair $(X^*,U^*)$ is the optimal solution of the linear program
  {   \begin{equation}
        \inf_{X,U,\alpha,\gamma} \gamma
      \end{equation}
     where $U\in\mathbb{R}^{n\times r}$, $\alpha\in\mathbb{R}$, $X\in\mathbb{R}^{n\times n}$ diagonal, $X\mathds{1}_n>0$,
\begin{equation}\label{eq:L1:cond1}
  \begin{array}{ccc}
  XA-UC +\alpha I_n\ge0, &&  XE-UF \ge0
  \end{array}
\end{equation}
and
  \begin{equation}\label{eq:L1:cond2}
  \begin{bmatrix}
    \mathds{1}_n\\
    1
  \end{bmatrix}^T\begin{bmatrix}
    XA-UC  &&     (XE-UF )\mathds{1}_p\\
    \mathds{1}_n^T && q-\gamma
  \end{bmatrix}<0
\end{equation}}
  %
\end{enumerate}
\end{theorem}
\begin{proof}
The proof is dual to the proof of Theorem 1 in \cite{Ebihara:12}. We can see first that, from Lemma \ref{lem:glab2}, the condition \eqref{eq:uni1} is satisfied if and only if
\begin{equation}
     L^*(\Sigma(\mathds{1}_{n}^T,\mathds{1}_{p}^T))\in\mathcal{L}^*(\Sigma(e_i^TM,e_i^TN))
      \end{equation}
holds for all $i=1,\ldots,q$ and all $(M,N)\in\mathbb{R}^{q\times n}_{\ge0}\times\mathbb{R}^{q\times p}_{\ge0}$. By virtue of Lemma \ref{lem:SISO}, this condition readily holds. Therefore, all we need to do is to compute the optimal gain $L^*(\Sigma(\mathds{1}_{n}^T,\mathds{1}_{p}^T))$. To see that this is performed by the linear program stated in the result, first note that the conditions $ XA-UC +\alpha I_n\ge0$ and $XE-UF \ge0$ are equivalent to saying that the matrix $A-LC$ is Metzler and that the matrix $E-LC$ is nonnegative whenever considering the change of variables $U=XL$. Using the same change of variables, the condition \eqref{eq:L1:cond2} is equivalent to saying that
\begin{equation}
  \begin{bmatrix}
    \mathds{1}_n\\
    1
  \end{bmatrix}^T\begin{bmatrix}
    X & 0\\
    0 & 1
  \end{bmatrix}\begin{bmatrix}
    A-LC  &&     (E-LF )\mathds{1}\\
    \mathds{1}^T && q-\gamma
  \end{bmatrix}<0.
\end{equation}
which is, in turn, equivalent to saying that the $L_\infty$ of the system $\Sigma(\mathds{1}_{n}^T,\mathds{1}_{p}^T)$ is smaller than $\gamma>0$. Therefore, the optimal-gain $L^*(\Sigma(\mathds{1}_{n}^T,\mathds{1}_{p}^T))$ can be computed by solving the linear program stated in the result. The proof is complete.
\end{proof}

\begin{remark}
  It is important to stress here that the minimal $\gamma^*>0$ that is computed using the linear program in Theorem \ref{th:L1obs} is not uniform over the set of $M^-,M^+$ matrices. In this regard, if we are interested in the actual  $L_\infty$-gain for given matrices $M^-,M^+$, one has to first compute the optimal gain $L^*$ using Theorem \ref{th:L1obs} and then compute the actual-gain for the error model using Theorem \ref{th:Linf}. This will be illustrated in the example in Section \ref{sec:ex1}.
\end{remark}

\begin{remark}\label{rem:bounds}
As shown in \cite{Briat:11h}, it is also possible to incorporate explicit bounds on the entries of $L$ in a non-conservative way. To this aim, let us define the matrices $L_{min},L_{max}\in\mathbb{R}^{n\times r}$ verifying $L_{min}\le L_{max}$, componentwise. Assume now that we would like to design a gain $L$ satisfying $L_{min}\le L\le L_{max}$, then it is necessary and sufficient to adjoin to the optimization problem of Theorem \ref{th:L1obs} the linear constraint $XL_{min}\le U\le XL_{max}$. Interestingly, when the $(i,j)$-th entries of $L_{min}$ and $L_{max}$ are identical, then the $(i,j)$'s entry of $L$ will be set to this specific value. This may be used to impose a structural constraint on the matrix $L$, i.e. some entries set to 0.
\end{remark}

\begin{remark}\label{rem:relax}
  In the case where the relaxed observer \eqref{eq:relaxobs} is considered, then an optimal gain $L^*$ can be computed by substituting $XE-UF$ by the identity matrix $I$ in Theorem \ref{th:L1obs}.
\end{remark}

\subsection{Extensions of the results}

We provide here some interesting extensions of the main result.

\subsubsection{Continuous-time systems with delay}

The results can also be straightforwardly extended to systems with delays \cite{Haddad:04,Briat:11h,Briat:book1} since it is known that a linear positive system with delay is stable if and only if the system with zero-delay is stable. By exploiting this connection with nondelayed systems, we can apply the obtained results to time-delay systems of the form:
\begin{equation}\label{eq:delaysyst}
  \begin{array}{rcl}
    \dot{x}(t)&=&Ax(t)+A_hx(t-h)+Ew(t)\\
    y(t)&=&Cx(t)+C_hx(t-h)+Fw(t)\\
    x(s)&=&\varphi(s),\ s\in[-h,0]
  \end{array}
\end{equation}
where $h\ge0$ is the constant delay and $\varphi\in C([-h,0],\mathbb{R}^n)$ is the initial condition. We then propose the following extension to the interval observer \eqref{eq:obs}:
\begin{equation}\label{eq:obsdelay}
\begin{array}{rcl}
      \dot{x}^\bullet(t)&=&Ax^\bullet(t)+A_hx^\bullet(t-h)+Ew^\bullet(t)+Lm^\bullet(t)\\
      m^\bullet(t)&=&y(t)-C x^\bullet(t)-C_h x^\bullet(t-h)-F w^\bullet(t)\\
      x^\bullet(0)&=&x_0^\bullet
\end{array}
\end{equation}
where $\bullet\in\{-,+\}$. The observed outputs $\zeta^\bullet$ are defined as in the delay-free case; i.e. $\zeta^\bullet=M^\bullet e^\bullet$. This leads us to the following result:

\begin{proposition}\label{cor:delay}
The following statements are equivalent:
\begin{enumerate}
  \item There exists an optimal $L_\infty$-to-$L_\infty$ interval-observer of the form \eqref{eq:obsdelay} for the system \eqref{eq:delaysyst}.
  \item We have that
  \begin{equation}\label{eq:uni1delay}
     L^*(\Sigma(\mathds{1}_{n}^T,\mathds{1}_{p}^T))\in\mathcal{L}^*(\Sigma(M,N))
      \end{equation}
  for all $(M,N)\in\mathbb{R}^{q\times n}_{\ge0}\times\mathbb{R}^{q\times p}_{\ge0}$. Moreover, the optimal gain is given by $L^*(\Sigma(\mathds{1}_{n}^T,\mathds{1}_{p}^T))=(X^*)^{-1}U^*$ where the pair $(X^*,U^*)$ is the optimal solution of the linear program
  \begin{equation}
        \inf_{X,U,\alpha,\gamma} \gamma
      \end{equation}
     where $U\in\mathbb{R}^{n\times r}$, $\alpha\in\mathbb{R}$, $X\in\mathbb{R}^{n\times n}$ diagonal, $X\mathds{1}_n>0$, ${XA-UC +\alpha I_n\ge0}$, ${XE-UF \ge0}$, ${XA_h-UC_h\ge0}$
and
  \begin{equation}\label{eq:L1:cond2delay}
  \begin{bmatrix}
    \mathds{1}_n\\
    1
  \end{bmatrix}^T\begin{bmatrix}
    \Psi  &&     (XE-UF)\mathds{1}_p\\
    \mathds{1}_n^T && q-\gamma
  \end{bmatrix}<0
\end{equation}
where $\Psi=X(A+A_h)-U(C+C_h)$.
\end{enumerate}
\end{proposition}
\begin{proof}
  This result is a consequence of Theorem \ref{th:L1obs} and the results on the $L_\infty$-gain of linear positive systems with delays; see e.g. \cite{Shen:14,Shen:15}.
\end{proof}

Note that the above result can be straightforwardly adapted to the cases of multiple constant delays, time-varying delays and distributed delays \cite{Shen:14}. This is omitted for brevity.

\subsubsection{Discrete-time systems}

Similarly to as in the continuous-time case, we can define the  $\ell_\infty$-gain of an LTI discrete-time system $G$ as
\begin{equation}
  ||G||_{\ell_\infty-\ell_\infty}:=\sup_{||w||_{\ell_\infty}=1}||Gw||_{\ell_\infty}
\end{equation}
where $\textstyle||w||_{\ell_\infty}:=\sup_{k\in\mathbb{N}}||w(k)||_\infty$. We will also work here with the following discrete-time counterpart of \eqref{eq:systy}
\begin{equation}\label{eq:DTsyst}
  \begin{array}{rcl}
    x(k+1)&=&A_dx(k)+E_dw(k),\     x(0)=x_0\\
    y(k)&=&C_dx(k)+F_dw(k)
  \end{array}
\end{equation}
for which we will consider the following interval-observer:
\begin{equation}\label{eq:obsDT}
\begin{array}{rcl}
      x^\bullet(k+1)&=&A_dx^\bullet(k)+E_dw^\bullet(k)+Lm^\bullet(k)\\
m^\bullet(k)&=&y(k)-C_d x^\bullet(k)-F_d w^\bullet(k)\\
      x^\bullet(0)&=&x_0^\bullet
\end{array}
\end{equation}
where $\bullet\in\{-,+\}$. The observed outputs $\zeta^\bullet$ are defined as in the continuous-time case; i.e. $\zeta^\bullet=M^\bullet e^\bullet$. While continuous-time and discrete-time stability conditions for general linear systems are structurally quite different, the stability conditions are interdependent in the case of positive systems because of the fact that a nonnegative matrix $Z$ is Schur stable iff the Metzler matrix $Z-I$ is Hurwitz stable\footnote{This follows from the facts that $Z\ge0$ is Schur stable iff there exists a vector $\lambda>0$ such that $\lambda^T(Z-I)<0$. This is, in turn, equivalent to the Hurwitz stability of the Metzler matrix $Z-I$.}. From a direct inspection of the results of Theorem \ref{th:Linf} and its discrete-time analogue (see \cite{Naghnaeian:14}), we can immediately see that the $\ell_\infty$-gain of a positive LTI discrete-time system $(A_d,E_d,C_d,F_d)$  is equal to the $L_\infty$-gain of the positive LTI continuous-time system ${(A_d-I,E_d,C_d,F_d)}$, thereby allowing us to state the following result:
\begin{proposition}\label{cor:DT}
The following statements are equivalent:
\begin{enumerate}
  \item There exists an optimal $\ell_\infty$-to-$\ell_\infty$ interval-observer of the form \eqref{eq:obsDT} for the system \eqref{eq:DTsyst}. 
   \item We have that
  \begin{equation}\label{eq:uni1DT}
     L^*(\Sigma(\mathds{1}_{n}^T,\mathds{1}_{p}^T))\in\mathcal{L}^*(\Sigma(M,N))
      \end{equation}
  for all $(M,N)\in\mathbb{R}^{q\times n}_{\ge0}\times\mathbb{R}^{q\times p}_{\ge0}$. Moreover, the optimal gain is given by $L^*=(X^*)^{-1}U^*$ where the pair $(X^*,U^*)$ is the optimal solution of the linear program
%
  {\begin{equation}
        \inf_{X,U,\gamma} \gamma
      \end{equation}
     where $U\in\mathbb{R}^{n\times r}$, $X\in\mathbb{R}^{n\times n}$ diagonal, $X\mathds{1}_n>0$, ${XA_d-UC_d\ge0}$, ${XE_d-UF_d \ge0}$
and
  \begin{equation}\label{eq:L1:cond2DT}
  \begin{bmatrix}
    \mathds{1}_n\\
    1
  \end{bmatrix}^T\begin{bmatrix}
    X(A_d-I)-UC  &&     (XE-UF)\mathds{1}_p\\
    \mathds{1}_n^T && q-\gamma
  \end{bmatrix}<0.
\end{equation}}
%
\end{enumerate}
\end{proposition}


By combining the results in Proposition \ref{cor:delay} and Proposition \ref{cor:DT}, we can readily obtain a condition for the design of an optimal interval observer for discrete-time systems with delays.

\section{Examples}\label{sec:ex}

\subsection{Stable linear systems}\label{sec:ex1}

\noindent\textbf{Case 1.} Let us consider the system \eqref{eq:syst} with matrices $C = \begin{bmatrix}
  0 & 1
\end{bmatrix}$, $F = 1$,
\begin{equation}\label{eq:decoupled}
A = \begin{bmatrix}
  -2 &1\\
   3 &-5
\end{bmatrix}\ \textnormal{and}\ E =\begin{bmatrix}
  1\\2
\end{bmatrix}.
\end{equation}
Using Theorem \ref{th:L1obs}, we find the optimal observer gain $L^*=[1\ \ 2]^T$. 
Note, however, that the computed minimum value for $\gamma^*$ is the one corresponding to the matrices $M^+=M^-=\mathds{1}^T_n$. Choosing then $M^+=M^-=I_n$ and computing the actual $L_\infty$-gain (using Theorem  \ref{th:Linf}) with the previously computed observer matrix gain $L^*$ gives an $L_\infty$-gain that is nearly 0, which is consistent with the fact that $E-L^*F=0$ (perfect decoupling). Using the initial conditions $x_0=[-1\ \ 2]^T$, $x_0^-=[-5\ \ -5]^T$, $x_0^+=[5\ \ 5]^T$, the persistent input $w(t)=\sin(t)$ and the bounds $w^-(t)=-1$ and $w^+(t)=1$, yield the trajectories depicted in Fig.~\ref{fig:stabdec} where we can see that the trajectories of the observers converge from above and below to the actual state trajectory of the system.

\noindent\textbf{Case 2.} Changing now the matrix $A$ to
\begin{equation}\label{eq:ndecoupled}
  A=\begin{bmatrix}
  -2 &-1\\ 3& -5
\end{bmatrix}
\end{equation}
and solving for the observer gain yields the value $L^*=[-1\ \ 2]^T$. 
Solving for the minimal $\gamma$ with $M^-=M^+=I$ in Theorem \ref{th:Linf} yields the value $\gamma^*=1.4304$, showing that decoupling is no longer possible if positivity is to be preserved. Using the same scenario as for the previous system yields the trajectories depicted in Fig.~\ref{fig:stabndec} where we can see that, in spite of the observation error, the optimal design of the observers allow for an accurate estimation of the interval.
%
\begin{figure}[H]
  \centering
\includegraphics[width=0.5\textwidth]{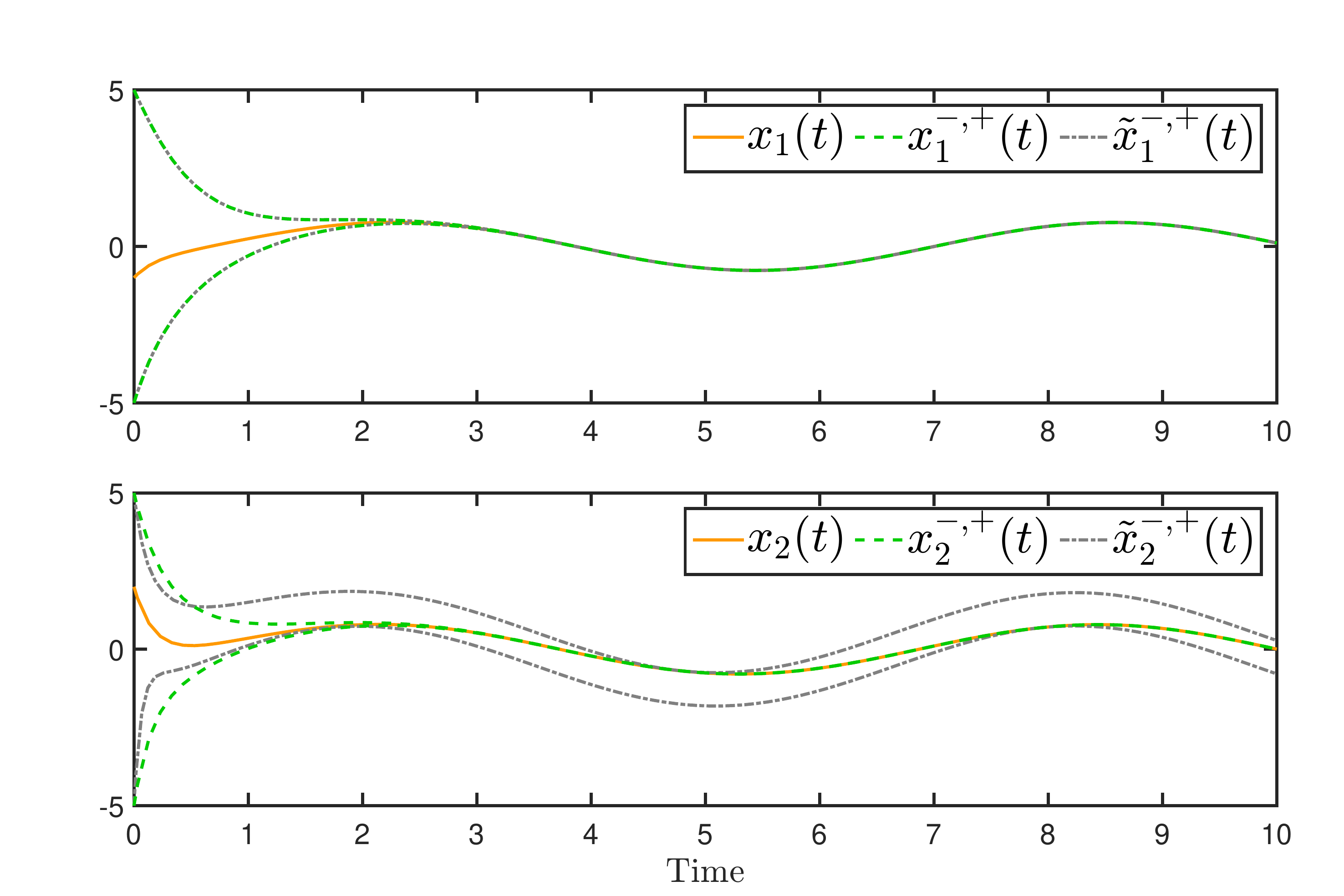}
  \caption{Trajectories of the system \eqref{eq:decoupled} and its optimal $L_\infty$-to-$L_\infty$ interval observers \eqref{eq:obs} and \eqref{eq:relaxobs}.}\label{fig:stabdec}
\end{figure}


\begin{figure}
  \centering
  \includegraphics[width=0.5\textwidth]{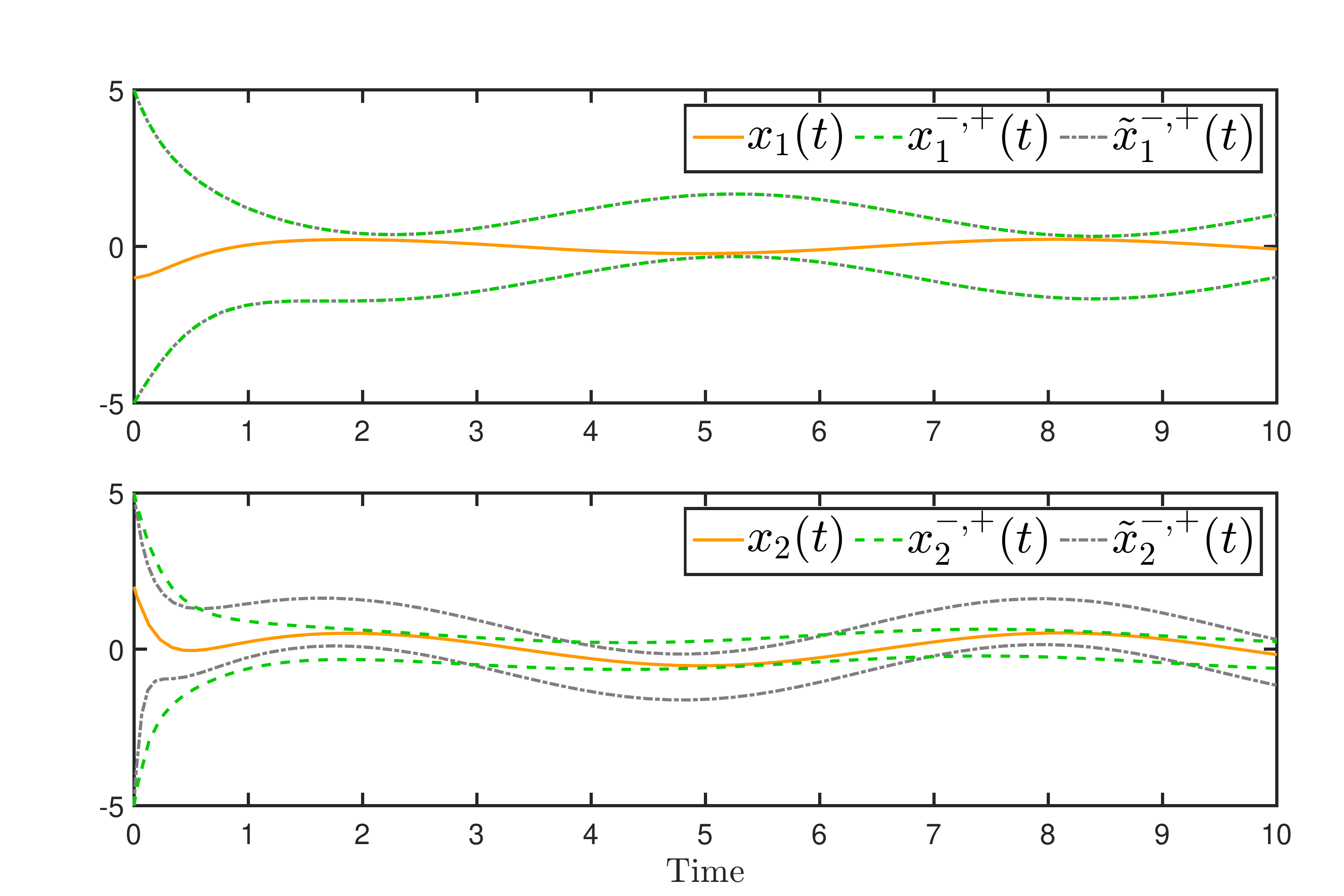}
  \caption{Trajectories of the system \eqref{eq:ndecoupled} and its optimal $L_\infty$-to-$L_\infty$ interval observer \eqref{eq:obs} and \eqref{eq:relaxobs}.}\label{fig:stabndec}
\end{figure}

\noindent\textbf{Case 3.} Finally, if we consider the matrix $A$ in \eqref{eq:ndecoupled} and the matrix $E=[1\ \ -6]^T$, then the optimization problem in Theorem \ref{th:L1obs} fails, meaning that an interval observer of the form  \eqref{eq:obs} does not exist for this system. This comes from the fact that $E-LF\ge0$, $L=[\ell_1\ \ell_2]^T$, if and only if $\ell_2\le -6$ and that a necessary condition for the matrix $A-LC$ be Hurwitz stable is $\ell_2>-5$.

\noindent\textbf{Comparison with the observer \eqref{eq:relaxobs}.} We now compare the results obtained with the observer \eqref{eq:obs} to the results obtained with the relaxed observer \eqref{eq:relaxobs}. Using Theorem \ref{th:L1obs} and Remark \ref{rem:relax}, we get that $L^*=[1\ 10]^T$ in \textbf{Case 1} and $L^*=[-1\ 10]^T$ in \textbf{Case 2} and \textbf{Case 3}. It is readily found that $A-L^*C$ is the same for all three cases and, therefore, that $\gamma^*=0.501$. We can see in Fig.\ref{fig:stabdec} that the observer \eqref{eq:obs} performs much better than the relaxed observer \eqref{eq:relaxobs} as it allows for decoupling through the consideration of the input matrix $E-LF$. In Fig.~\ref{fig:stabndec}, we can also see that the interval observer  \eqref{eq:obs} gives better results than the relaxed observer \eqref{eq:relaxobs} since the maximal error is lower, even though the theoretical attenuation level $\gamma$ is higher. A main issue with the relaxed observer is that the disturbance vector is artificially augmented from dimension $p$ to $n\ge p$. In this regard, the computed optimal gain of the relaxed observer \eqref{eq:relaxobs} is clearly suboptimal for the original problem. The advantage of the relaxed observer is shown in Fig.~\ref{fig:stabndec_nE} where interval observation is performed on the system treated in \textbf{Case 3} for which no interval observer of the form \eqref{eq:obs} exists.

%

\begin{figure}
  \centering
  \includegraphics[width=0.5\textwidth]{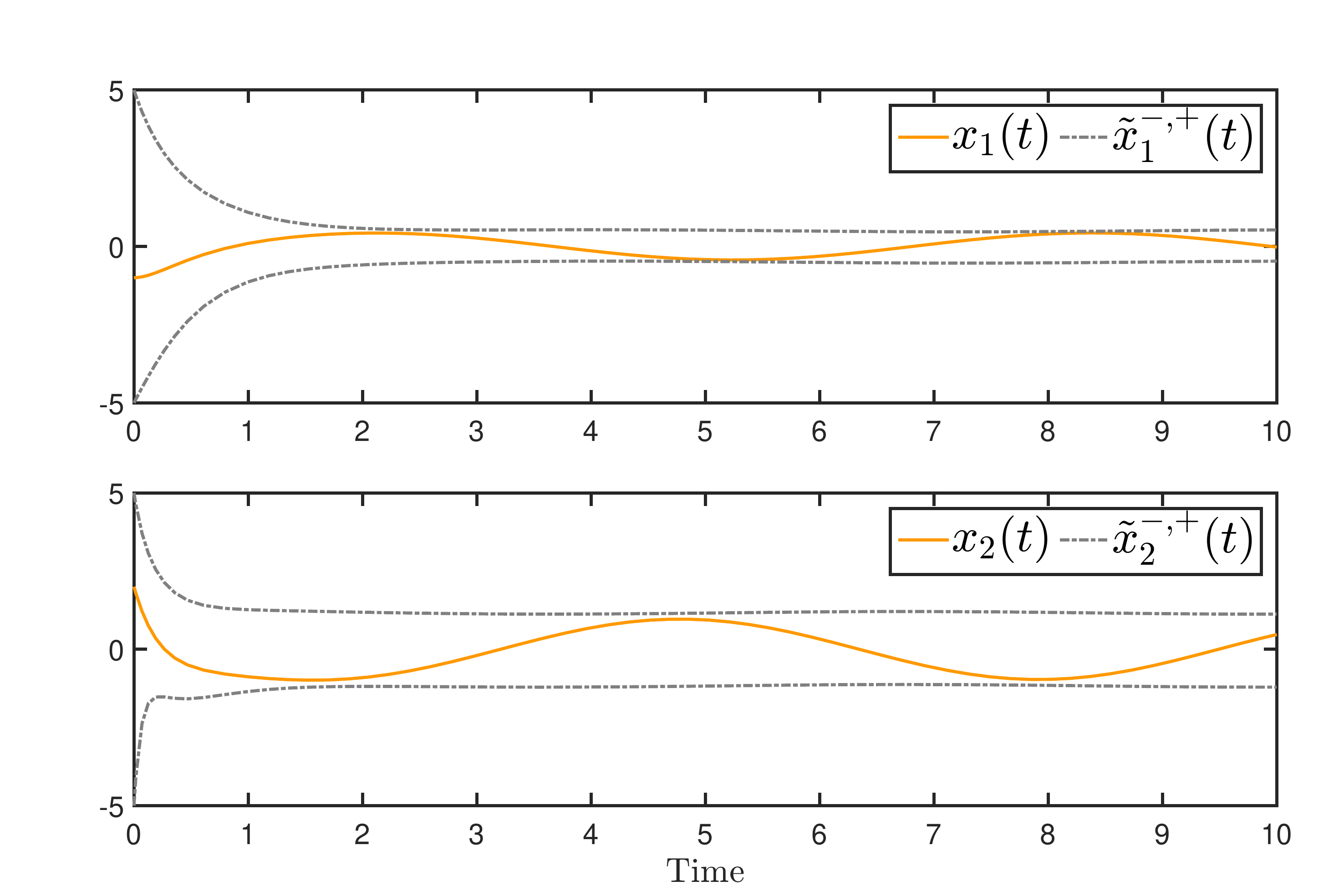}
  \caption{Trajectories of the system \eqref{eq:ndecoupled} with $E=[0\ \ -6]^T$ and its optimal $L_\infty$-to-$L_\infty$ interval observer \eqref{eq:relaxobs}.}\label{fig:stabndec_nE}
\end{figure}

%
%

\subsection{A nonlinear system}\label{sec:ex2}

Let us consider now the following 3-stage population model \cite{Gouze:00} given by
\begin{equation}\label{eq:population}
  \begin{array}{rcl}
    \dot{x}_1(t)&=&-\beta_1x_1(t)+\dfrac{a(t)x_3(t)}{b+x_3(t)}\\
    \dot{x}_2(t)&=&\alpha_1x_1(t)-\beta_2x_2(t)\\
    \dot{x}_3(t)&=&\alpha_2x_2(t)-\beta_3x_3(t)
  \end{array}
\end{equation}
where $\alpha_1,\alpha_2,\beta_1,\beta_2,\beta_3$ and $b$ are positive parameters. The function $a(t)$ is assumed to be continuous and to satisfy $a(t)\in[a^-,a^+]$, for some $0<a^-<a^+$. These system may have either one or two equilibrium points depending on the values of the parameters. The 0-equilibrium point always exists whereas the positive one exists only for some values of the parameters.

Let us assume that we measure $x_3(t)$, i.e. $y(t)=x_3(t)$. We then consider the interval observer \eqref{eq:obs} with
\begin{equation}\label{eq:Aex}
  \hspace{-2pt}A=\begin{bmatrix}
    -\beta_1 &  0 & 0\\
    \alpha_1 & -\beta_2&0\\
    0 & \alpha_2 & -\beta_3
  \end{bmatrix},\ E=\begin{bmatrix}
    1\\
    0\\
    0
  \end{bmatrix}\ \textnormal{and}\ C=\begin{bmatrix}
    0\\
    0\\
    1
  \end{bmatrix}^T.
\end{equation}
Note that $A$ is Metzler and Hurwitz stable. Define, moreover, the signals $w(t)=a(t)h(x_3(t))$, $w^-(t)=a^-h(x_3(t))$, $w^+(t)=a^+h(x_3(t))$ and $h(x_3)=x_3/(x_3+b)$. Note that these signals are not persistent when the 0-equilibrium is asymptotically stable. We then have the following result:
\begin{proposition}
Let $M^+=M^-=I_n$, then the minimal $L_\infty$-gain achieved by an interval observer of the form \eqref{eq:obs} for the system \eqref{eq:syst2}-\eqref{eq:Aex} is given by
\begin{equation}
  \max\left\{\dfrac{1}{\beta_1},\dfrac{\alpha_1}{\beta_1\beta_2}\right\}
\end{equation}
and is achieved with any optimal gain $L$ of the form
  $L=\begin{bmatrix}
    0 & 0 & \ell_3
  \end{bmatrix}^T$ where $\ell_3$ is such that
  \begin{equation}
    \ell_3>\alpha_2\max\left\{1,\ \dfrac{\alpha_1}{\beta_2}\right\}-\beta_3.
  \end{equation}
\end{proposition}
\vspace{-4mm}
\begin{proof}
We show first that $\ell_1=\ell_2=0$. To this aim, let us note that when $\ell_1=\ell_2=0$ and $\ell_3>0$, then the matrix $A-LC$ is Hurwitz stable. Moreover, for the matrix $A-LC$ to be Metzler, we need that both $\ell_1$  and $\ell_2$ be nonpositive. We also note that when $\ell_3$ is given, then the matrix $A-LC$ remains stable for some sufficiently small $|\ell_1|$ and $|\ell_2|$. We show now that for any $\ell_3$ such that the matrix $A-LC$ is Hurwitz stable, negative values for $\ell_1$ and $\ell_2$ will increase the norm of the system. Let us define $LC=:\ell_3vv^T+uv^T$ where $u=[\ell_1\ \ell_2\ 0]^T$ and $v=[0\ 0\ 1]^T$. Then, we have that $A-LC=A-\ell_3vv^T-uv^T$ and using the Sherman-Morrison formula, we get that
\vspace{-3mm}
\begin{equation*}
   -(A-LC)^{-1}E=-\left(\tilde{A}^{-1}+\dfrac{\tilde{A}^{-1}uv^T\tilde{A}^{-1}}{1-v^T\tilde{A}^{-1}u}\right)E\ge -\tilde{A}^{-1}E
   \vspace{-3mm}
\end{equation*}
where  $\tilde A=A-\ell_3vv^T$. This inequality follows from the facts that the numerator of the fraction above is nonpositive (recall that $\tilde{A}^{-1}\le$ and that $\ell_1,\ell_2\le0$), that its denominator is positive for some sufficiently small $|\ell_1|$ and $|\ell_2|$ and that $E$ is nonnegative. This implies that when $|\ell_1|$ and $|\ell_2|$ increase, then the gain increases as well. Therefore we need to set $\ell_1=\ell_2=0$, which leads to the expression:
%
\vspace{-2mm}
\begin{equation}\label{eq:kdslmldksmkd}
  \max\{-(A-LC )^{-1}E\}=\dfrac{1}{\beta_1}\max\left\{1,\dfrac{\alpha_1}{\beta_2},\dfrac{\alpha_1\alpha_2}{\beta_2(\beta_3 + \ell_3)}\right\}.
\end{equation}
We can see that by choosing a sufficiently large $\ell_3$, then we can act on the last entry and make it smaller than the max of the others. Simple calculations then yield the result.
\end{proof}

This example is interesting numerically speaking as it illustrates some difficulties arising when no minimum exists and we have an infimum instead, resulting then in very large gain values. Indeed, remember that the gain of the observer is computed for $M=\mathds{1}^T_n$ and, therefore, the $L_\infty$-gain for the current system consists of the sum of the possible bounds in \eqref{eq:kdslmldksmkd}, which has no minimum because of the presence of $\ell_3$ at the denominator of one of these values. In order to overcome this situation, we can consider bounds on the observer gain as stated in Remark \ref{rem:bounds}. So, the idea is to compute an observer gain that will minimize the $L_\infty$-gain of the transfer of the disturbance to the observed outputs while satisfying the constraint $-5\mathds{1}_3\le L\le 5\mathds{1}_3$. Using the parameters $\beta_1=2$, $\beta_2=2$, $\beta_3=3$, $\alpha_1=3$ and $\alpha_2=4$, we find that the optimal gain is given by $L^*=[0\ 0\ 5]^T$ and that $\gamma^*=\alpha_1/(\beta_1\beta_2)=3/4$. A simulation is then performed with the parameters $b=1$, $a(t)=(a^-+a^+)/2+(a^+-a^-)\sin(0.1t)/2$, $a^-=1$, $a^+=2$, and the initial conditions $x_0=[0.1\ \ 0\ \ 0]^T$, $x_0^-=[0.01\ \ 0\ \ 0]^T$, $x_0^+=[0.6\ \ 0.8\ \ 1.1]^T$. The results are depicted in Fig.~\ref{fig:population} where we can see that the interval observer is able to provide a fairly good estimation for the trajectories of the system. 

\begin{figure}[h!]
  \centering
  \hspace{-13pt}\includegraphics[width=0.5\textwidth]{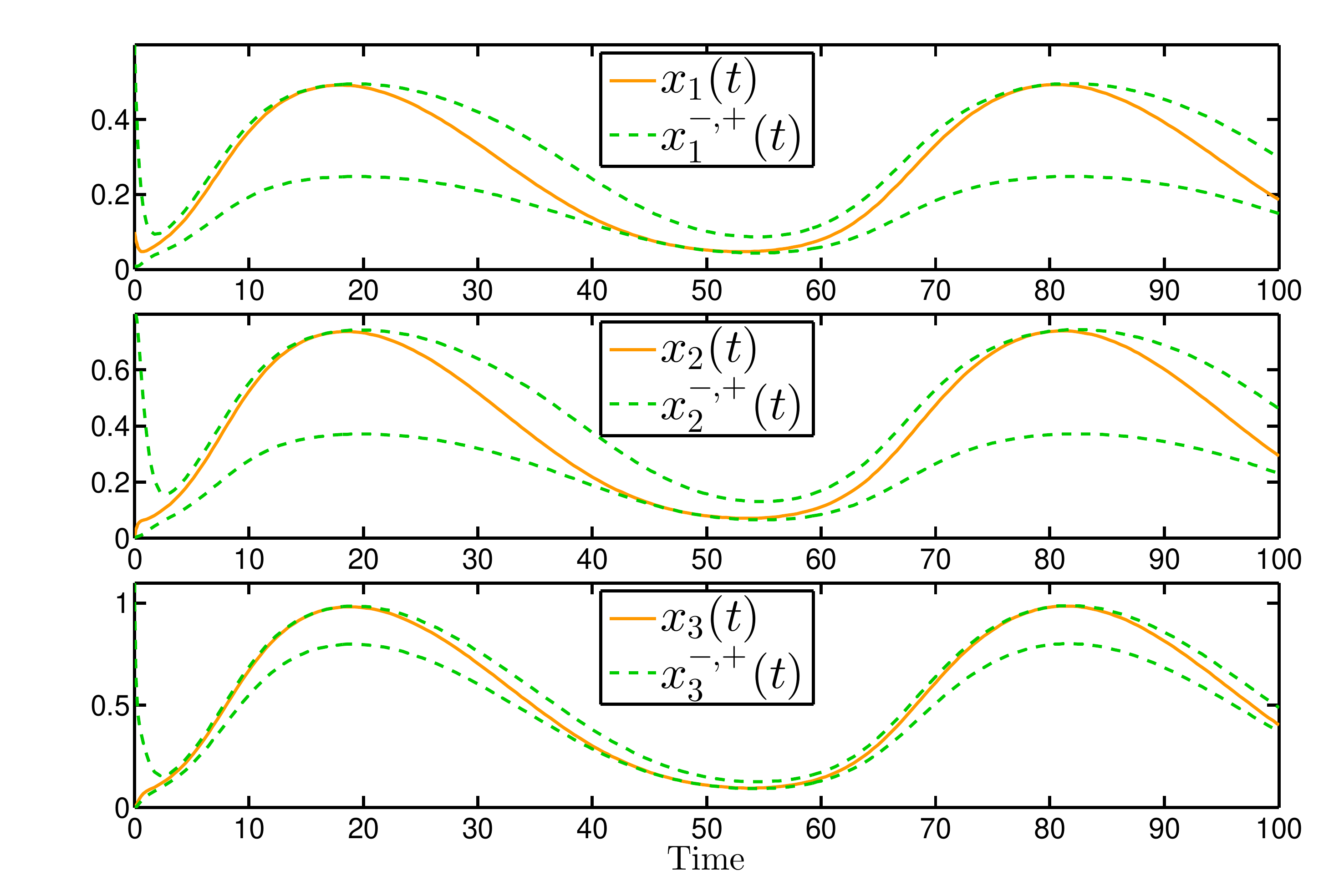}
  \caption{Trajectories of the population model \eqref{eq:population} and its optimal $L_\infty$-to-$L_\infty$ interval observer.}\label{fig:population}
\end{figure}

\section{Conclusion}

Linear programming conditions have been obtained for the design of a class of optimal $L_\infty$-to-$L_\infty$  interval observers for continuous-time systems, discrete-time systems and systems with delays. It is shown that the optimal observer is uniform over the set of matrices mapping the estimation errors to the observed outputs. The approach is versatile and should be applicable to more complex systems such as uncertain systems \cite{Colombino:15}. However, the consideration of more general observers, such as those considered in \cite{Blanchini:12}, that can be applied to a wider class of systems seems to be more difficult. It is also unclear, for the moment, how to non-conservatively extend the current approach to the design observer-based controllers. A final interesting question is whether the consideration of higher-order observers would improve the performance of the observers. These questions are left for future research.

\bibliographystyle{IEEEtran}

\end{document}